\documentclass[11pt,a4paper]{amsart}
\usepackage{amssymb, amstext, amscd, amsmath, color}
\usepackage{graphicx}

\usepackage{url}

\usepackage{tikz,enumerate}
\usepackage{pgfplots}
\usepackage{url}

\usepackage{hhline}
\usepackage{hyperref}

\pgfplotsset{compat=1.10}
\begin{document}

\newtheorem{thm}{Theorem}[section]
\newtheorem{cor}[thm]{Corollary}
\newtheorem{prop}[thm]{Proposition}
\newtheorem{lem}[thm]{Lemma}
%
\theoremstyle{definition}
\newtheorem{rem}[thm]{Remark}
\newtheorem{defn}[thm]{Definition}
\newtheorem{note}[thm]{Note}
\newtheorem{eg}[thm]{Example}
\newcommand{\Prf}{\noindent\textbf{Proof.\ }}
\newcommand{\bx}{\hfill$\blacksquare$\medbreak}
\newcommand{\upbx}{\vspace{-2.5\baselineskip}\newline\hbox{}%
\hfill$\blacksquare$\newline\medbreak}
\newcommand{\eqbx}[1]{\medbreak\hfill\(\displaystyle #1\)\bx}

\newcommand{\FFock}{\mathcal{F}}
\newcommand{\kil}{\mathsf{k}}
\newcommand{\Hil}{\mathsf{H}}
\newcommand{\hil}{\mathsf{h}}
\newcommand{\Kil}{\mathsf{K}}
\newcommand{\Real}{\mathbb{R}}
\newcommand{\Rplus}{\Real_+}

\newcommand{\bC}{{\mathbb{C}}}
\newcommand{\bD}{{\mathbb{D}}}
\newcommand{\bK}{{\mathbb{K}}}
\newcommand{\bN}{{\mathbb{N}}}
\newcommand{\bQ}{{\mathbb{Q}}}
\newcommand{\bR}{{\mathbb{R}}}
\newcommand{\bT}{{\mathbb{T}}}
\newcommand{\bX}{{\mathbb{X}}}
\newcommand{\bZ}{{\mathbb{Z}}}
\newcommand{\bH}{{\mathbb{H}}}
\newcommand{\BH}{{\B(\H)}}
\newcommand{\bsl}{\setminus}
\newcommand{\ca}{\mathrm{C}^*}
\newcommand{\cstar}{\mathrm{C}^*}
\newcommand{\cenv}{\mathrm{C}^*_{\text{env}}}
\newcommand{\rip}{\rangle}
\newcommand{\ol}{\overline}
\newcommand{\td}{Widetilde}
\newcommand{\Wh}{Widehat}
\newcommand{\sot}{\textsc{sot}}
\newcommand{\Wot}{\textsc{wot}}
\newcommand{\Wotclos}[1]{\ol{#1}^{\textsc{wot}}}
 \newcommand{\A}{{\mathcal{A}}}
 \newcommand{\B}{{\mathcal{B}}}
 \newcommand{\C}{{\mathcal{C}}}
 \newcommand{\D}{{\mathcal{D}}}
 \newcommand{\E}{{\mathcal{E}}}
 \newcommand{\F}{{\mathcal{F}}}
 \newcommand{\G}{{\mathcal{G}}}
\renewcommand{\H}{{\mathcal{H}}}
 \newcommand{\I}{{\mathcal{I}}}
 \newcommand{\J}{{\mathcal{J}}}
 \newcommand{\K}{{\mathcal{K}}}
\renewcommand{\L}{{\mathcal{L}}}
 \newcommand{\M}{{\mathcal{M}}}
 \newcommand{\N}{{\mathcal{N}}}
\renewcommand{\O}{{\mathcal{O}}}
\renewcommand{\P}{{\mathcal{P}}}
 \newcommand{\Q}{{\mathcal{Q}}}
 \newcommand{\R}{{\mathcal{R}}}
\renewcommand{\S}{{\mathcal{S}}}
 \newcommand{\T}{{\mathcal{T}}}
 \newcommand{\U}{{\mathcal{U}}}
 \newcommand{\V}{{\mathcal{V}}}
 \newcommand{\W}{{\mathcal{W}}}
 \newcommand{\X}{{\mathcal{X}}}
 \newcommand{\Y}{{\mathcal{Y}}}
 \newcommand{\Z}{{\mathcal{Z}}}

\newcommand{\fA}{{\mathfrak{A}}}
\newcommand{\fB}{{\mathfrak{B}}}
\newcommand{\fC}{{\mathfrak{C}}}
\newcommand{\fD}{{\mathfrak{D}}}
\newcommand{\fE}{{\mathfrak{E}}}
\newcommand{\fF}{{\mathfrak{F}}}
\newcommand{\fG}{{\mathfrak{G}}}
\newcommand{\fH}{{\mathfrak{H}}}
\newcommand{\fI}{{\mathfrak{I}}}
\newcommand{\fJ}{{\mathfrak{J}}}
\newcommand{\fK}{{\mathfrak{K}}}
\newcommand{\fL}{{\mathfrak{L}}}
\newcommand{\fM}{{\mathfrak{M}}}
\newcommand{\fN}{{\mathfrak{N}}}
\newcommand{\fO}{{\mathfrak{O}}}
\newcommand{\fP}{{\mathfrak{P}}}
\newcommand{\fQ}{{\mathfrak{Q}}}
\newcommand{\fR}{{\mathfrak{R}}}
\newcommand{\fS}{{\mathfrak{S}}}
\newcommand{\fT}{{\mathfrak{T}}}
\newcommand{\fU}{{\mathfrak{U}}}
\newcommand{\fV}{{\mathfrak{V}}}
\newcommand{\fW}{{\mathfrak{W}}}
\newcommand{\fX}{{\mathfrak{X}}}
\newcommand{\fY}{{\mathfrak{Y}}}
\newcommand{\fZ}{{\mathfrak{Z}}}

\newcommand{\sgn}{\operatorname{sgn}}
\newcommand{\rank}{\operatorname{rank}}
\newcommand{\supp}{\operatorname{supp}}
\newcommand{\dist}{\operatorname{dist}}
\newcommand{\Aut}{\operatorname{Aut}}
\newcommand{\Aff}{\operatorname{Aff}}
\newcommand{\Cknet}{{\mathcal{C}_{{\rm Knet}}}}
\newcommand{\Ckag}{{\mathcal{C}_{{\rm kag}}}}
\newcommand{\GL}{\operatorname{GL}}
\newcommand{\spn}{\operatorname{span}}

\newcommand{\ul}{\underline}

\title[Non-Euclidean braced grids]{Non-Euclidean braced grids}


\author[Stephen Power]{Stephen Power}

\address{Dept.\ Math.\ Stats.\\ Lancaster University\\
Lancaster LA1 4YF \\U.K. }
\email{s.power@lancaster.ac.uk}


\thanks{
Key words and phrases: Braced grids, non-Euclidean norms, rigidity, flexibility\hfill 
}

\date{}

\begin{abstract} 
Necessary and sufficient conditions are obtained for the 
infinitesimal rigidity of braced grids in the plane with respect to non-Euclidean norms. Component rectangles of the grid may carry 0, 1 or 2 diagonal braces, and the combinatorial part of the conditions is given in terms of a matroid for the bicoloured bipartite multigraph defined by the braces. 
\end{abstract}

\maketitle

\section{Introduction}
In this note we consider how to rigidly brace an $m \times n$ grid of flexible squares when distances are measured with respect to a general  norm. The characterisation for the Euclidean norm, due to Bolker and Crapo \cite{bol-cra} in 1977, is well-known: bracing some of the squares, by adding a diagonal bar, gives an infinitesimally rigid bar-joint framework if and only if the subgraph of the complete bipartite graph $K_{m,n}$ determined by the braced squares is connected and spanning.
In the anisotropic non-Euclidean setting novel phenomena appear. A singly braced square is infinitesimally flexible and when it is doubly braced its infinitesimal rigidity may depend on its inclination relative to the principal axes. 
The natural braces graph is therefore a subgraph of the bicoloured bipartite multigraph $K^2_{m,n}$, in which each edge of $K_{m,n}$ is doubled and carries a distinct colour, blue or red. The colour of the edge corresponds to the translation class of the represented brace.

Let $\B_{\rm max}$ be the set of $2mn$ possible diagonal braces
that could be added to an $m \times n$ grid of squares  framework $\G$, and let $\G(\B)$ be the braced grid framework determined by a braces set 
$\B\subseteq \B_{\rm max}$. The fully doubly braced grid  $\G(\B_{\rm max})$ may be infinitesimally flexible in $(\bR^2,\|\cdot\|)$ for certain values of angular inclination relative to the $x$-axis. 
When the norm is differentiable and strictly convex we show how these exceptional values are determined by the geometry of the unit sphere $\{(x,y):\|(x,y)\|=1\}$ and its set of tangents. See Lemma \ref{l:braceparams}. 
Additionally, we obtain the following complementary result, an analogue of the Bolker-Crapo theorem. 

A cycle of edges $e_1, \dots ,e_{2k}$ in $K^2_{m,n}$  is said to have a \emph{dependent colouring}, or to be \emph{dependent}, if the number of blue edges in
$\{e_1, e_3, \dots ,e_{2k-1}\}$ is equal to the number of blue edges in
$\{e_2, e_4, \dots ,e_{2k}\}$, otherwise the cycle is said to be \emph{independent}, with a \emph{independent colouring}.

\begin{thm}\label{t:nonEthm} 
Let $\G(\B)$ be an $m\times n$ braced grid bar-joint framework in $\bR^2$ and let $\|\cdot\|$ be a differentiable, strictly convex non-Euclidean norm.
Then the following are equivalent.

(i) $\G(\B)$ is infinitesimally rigid in $(\bR^2,\|\cdot\|)$.

(ii) $\G(\B_{\rm max})$ is infinitesimally rigid in $(\bR^2,\|\cdot\|)$ and the bicoloured braces graph of  $\G(\B)$ is a spanning subgraph of $ K^2_{m,n}$ with an independent cycle in each path-connected component.
\end{thm}

A simple graph is a \emph{cycle-rooted tree} if there is an edge whose removal gives a tree, and is a \emph{cycle-rooted forest} if each connected components is a cycle-rooted tree.
We extend this tree and forest terminology to bicoloured multigraphs $G$ whose monochrome subgraphs are simple. Thus $G$ is a cycle-rooted tree if there is an edge whose removal gives a tree.
The graph condition in (ii) means that the braces graph contains a spanning subgraph which is an \emph{independent} cycle-rooted forest in the sense that it is a cycle-rooted forest and the unique cycle in each component is independent. Figure \ref{f:nonEgrid} shows a bracing pattern for which the braces graph is actually equal to such a spanning subgraph.
In such cases if the maximally braced graph is infinitesimally rigid then $\G(\B)$ is {minimally infinitesimally rigid}, or isostatic.

\begin{center}
\begin{figure}[ht]
\centering
\includegraphics[width=4.5cm]{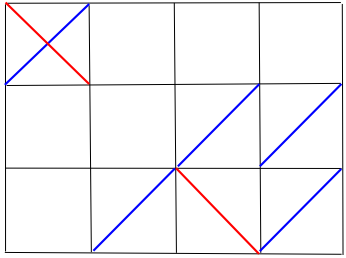}
\caption{A bracing pattern for which the braces graph is a spanning  cycle-rooted forest in $K_{4,3}^2$, where each cycle has an independent colouring.}
\label{f:nonEgrid}
\end{figure}
\end{center}

We also show that Theorem \ref{t:nonEthm} may be generalised to braced grid frameworks where the underlying grid has an irregular spacing, and where independence for cycles and forests is defined in terms of a gain graph formalism.

In the case of norms with 4-fold rotational symmetry, such as the classical norms $\|\cdot\|_p$, with $1 <p<\infty, p\neq 2$, we find that
the maximally braced square grids, $\G(\B_{\rm max})$, of any inclination, are infinitesimally flexible. On the other hand this is not so in the case of properly rectangular grids and the combinatorial condition (ii) applies. 

In Kitson and Power \cite{kit-pow} we began the analysis of the rigidity of bar-joint frameworks $(G,p)$ in non-Euclidean spaces $(\bR^d, \|\cdot\|)$. In particular for the non-Euclidean norms $\|\cdot\|_p, 1 \leq q \leq \infty,$ we obtained analogues of the Laman/Pollaczek-Geiringer combinatorial characterisation of generic rigidity for the Euclidean plane. This has recently been generalised to arbitrary norms by Dewar \cite{dew}. See also Remark \ref{r:22tight}.
For the proof of Theorem \ref{t:nonEthm} however, nongeneric methods are required and we follow a similar path to Bolker and Crapo, relating infinitesimal rigidity to the maximal independent sets in the matroid of an appropriate stress-sheer matrix. The independent sets in this matroid are the independent cycle-rooted forests of $K_{m,n}^2$. 

In the final section we note that there are similar characterisations for infinite braced grids.


\section{Euclidean braced grids}\label{s:BCproof}

A bar-joint framework $\G=(G,p)$ in  $(\bR^d, \|\cdot\|_2)$ is a finite or countable simple graph $G=(V,E)$ together with a placement $p:V \to \bR^d$ of its vertices.
A (real) \emph{infinitesimal flex} of  $\G$  is a vector field $u:p(V) \to \bR^d$ which satisfies the first order flex condition for every bar. In terms of the standard inner product for $\bR^d$ this means that
\[
\langle u(p(v))-u(p(w)), p(v)-p(w)\rangle = 0, \quad \mbox{ for } vw \in E.
\]

A framework is \emph{infinitesimally rigid} if every infinitesimal flex is a \emph{rigid motion infinitesimal flex} \cite{asi-rot}. For an $m \times n$ braced grid framework the vector space of rigid motion flexes coincide with the space of infinitesimal flexes of the fully braced grid, and this is 3-dimensional being spanned by two infinitesimal translations and an infinitesimal rotation.

We first recall that the sufficiency of the braces graph condition for infinitesimal rigidity in the Bolker-Crapo characterisation is  straightforward and follows quickly from the following lemma.


Define a \emph{bar 4-cycle} for a braced grid to be a 4-cycle of bars in the associated unbraced grid framework. 


\begin{lem}\label{l:ABCDlemma}\label{f:2Dprinciple} Let $\G(\B)$ be a  braced $m \times n$ grid in $(\bR^2, \|\cdot\|_2)$ with a triple of braces corresponding to the squares with labels $(m_1, n_1), (m_2, n_1), (m_1, n_2)$. Then the restriction of an infinitesimal flex $z$ of $\G(\B)$ to the bar 4-cycle for the square  with label $(m_2, n_2)$ is a rigid motion flex.
\end{lem}

\begin{proof}Since $\|\cdot\|_2$ is isotropic we may assume that the boundary of the grid is parallel to the coordinate axes.
By adding a rigid motion flex we may assume that the restriction of $z$ to the braced bar 4-cycle for $(m_1, n_1)$ is zero. In view of the linear geometry of the grid the restriction of $z$ to the joints of the braced bar 4-cycle for $(m_2, n_1)$ is an infinitesimal vertical translation. Similarly, the restriction of $z$ to 
the braced bar 4-cycle for $(m_1, n_2)$ is a horizontal infinitesimal translation. By the linear geometry it follows now that at the 4 joints of the bar 4-cycle for $(m_2, n_2)$ the vectors of $z$ have horizontal and vertical components equal to these horizontal and vertical velocities. In particular the restriction of $z$ to the bar 4-cycle for $(m_2, n_2)$ is an infinitesimal translation, as required.
\end{proof}

The lemma implies that if the braces graph $H$ contains the path $(v_1, w_1), (v_2, w_1), (v_2, w_2)$ in $K_{m,n}$ then, in determining the infinitesimal flex space of the braced grid $\G$, we may assume that $H$ also contains $(v_1, w_2)$. Suppose then that $H$ is connected and spanning. Repeating this edge addition principle we can assume that the connected spanning graph $H$ is equal to $K_{m,n}$, the graph associated with the fully braced grid. Since this grid is infinitesimally rigid for the Euclidean norm the sufficiency direction follows. 

The necessity of the condition, that the braced grid is infinitesimally flexible if $H$ is not connected and spanning, is more subtle. It follows from the fact that the property of rigidity or flexibility of a braced grid is unchanged if one performs row and column permutations to change brace positions. This invariance becomes clear on expressing infinitesimal flexes of the unbraced grid in terms of sheering flexes. We now discuss this shift of viewpoint, which is one of main devices in the analysis of braced grids in 2 and 3 dimensions in Bolker-Crapo\cite{bol-cra} and in Bolker \cite{bol}.

\subsection{The stress-sheer matrix}
We adopt the following terminology and notation 
associated with the geometry and linear algebra of a finite grid framework. This will also be useful for non-Euclidean grids.

Let $C(m,n)$, the \emph{$m \times n$ grid}, be the closed subset of the plane which is the union of the boundaries of the squares
$[j,j+1]\times [k,k+1], 0\leq j \leq m, 0\leq k \leq n$. 
 
Let $L_x$ be the set of $n+1$ subsets $l=  [0,m]\times k$, the \emph{lines} of the grid in the $x$-direction. Define $L_y$ similarly and let $L= L_x\cup L_y$, the set of all lines for the grid.

Let $R=R_x\cup R_y$ be the set of $n + m$ \emph{ribbons} of $C(m,n)$. These are rectangles labelled by an adjacent pair of parallel boundary lines $l_1, l_2 \in L$, there being no repetitions if $m+n>1$.

The lines $l$ in $L= L_x\cup L_y$ determine linear subframeworks, or \emph{line frameworks}, of a braced grid framework $\G(\B)$ associated with $C(m,n)$.
For each line $l$ in $L$ let $u_l$ be the unique vector field which vanishes everywhere except on the joints of the line, where it has unit norm and positive direction parallel to the line. It is elementary to show that the set of these vector fields is a basis for the infinitesimal flex space of the unbraced $m\times n$  grid, $\G$ say,  associated with $C(m,n)$. In particular this vector space of flexes has dimension $m+n+2$. 

Let $\M$ (resp.  $\M(\B)$) be the infinitesimal flex space of the unbraced grid $\G$ (resp. $\G(\B)$.
Then
$\M$  is identifiable with the vector space $\bR^{L}$ of functions from $L$ to $\bR$.

Let $\S$ be the vector space $\bR^{R}$ which we refer to as the space of \emph{ribbon sheers}. A ribbon sheer is thus a scalar field on the set of ribbons. The terminology comes from the mechanical viewpoint that the difference of applied velocities along the bounding lines of a ribbon is a sheer. 

Let $\sigma$ be the vector space homomorphism from $\M$ to $\S$ given by a choice of signs, $\sgn(l;\rho)$, for the boundary lines the ribbons $\rho$, with 
\[
\sigma(u)(\rho) = 
\sgn(l_1;\rho)d_{l_1} + \sgn(l_2;\rho)d_{l_2},
\quad \mbox{ for }\rho \in R, ~~~ u = \sum_{l\in L} d_lu_l.
\]
We assume that for a vertical (resp. horizontal) ribbon
the boundary line $l_1$ with $\sgn(l_1;\rho)$ positive is the right hand (resp. lower) line. The other boundary line has negative sign.

Let $\M_{\rm trans}$ (resp. $\M_{\rm rig})$  be the  subspace of $\M$,  and of $\M(\B)$, consisting of translation flexes (resp. rigid motion flexes). 
Note in particular that $\sigma:\M \to \S$ is onto and $\ker \sigma =\M_{\rm trans}$. 

Let $L_x$ (resp. $L_y$) consist of the consecutive lines $l_0, l_1,\dots ,l_m$ (resp. $l_0', l_1',\dots ,l_n'$) and define
\[
u_{\rm rot} =  \sum_{k=0}^m ku_{l_k} + \sum_{k=0}^n -ku_{l_k'}
\]
Then $u_{\rm rot}\in \M(\B)$ is an infinitesimal rotation velocity field which fixes the south-west corner joint of the braced grid. Moreover we have
$\sigma(u_{\rm rot}) =  \ul{1}$, the ribbon sheer field $(1,1,\dots ,1)$ in $\bR^R$. This leads to the following lemma.


\begin{lem}\label{l:sigmaLemma}
Let $\G(\B)$ be a finite braced grid in $(\bR^2, \|\cdot\|_2)$.
Then  the space of rigid motion infinitesimal flexes, $\M_{\rm rig}$, is equal to $\{u: \sigma(u) \in \bR\ul{1}\}$.
\end{lem}

\begin{proof} Let $z$ be an infinitesimal flex with $\sigma(z)=\lambda\ul{1}$. Subtract $\lambda u_{\rm rot}$ from $z$ to obtain $z'$ with $\sigma(z')=0$. For some set of coefficients $d_l$,  
\[
z' = \sum_{l\in L_x} d_lu_l + \sum_{l\in L_y} d_lu_l. 
\]
The condition $\sigma(z')(\rho)=0$ for every ribbon $\rho$ implies that the coefficients are equal for ${l\in L_x}$ and are also equal for ${l\in L_y}$. Thus $z'$ is an infinitesimal translation and the lemma follows.
\end{proof}

\medskip

The \emph{stress-sheer matrix} of $\G(\B)$, denoted $SS(\B)$, is the $|\B|\times |R|$ matrix where the row labelled by $b\in \B$ has zero entries except for entries of +1 and -1 for the columns for the ribbons for $b$ in the $x$- and $y$-directions, respectively.
The row for the brace $b$ defines a brace functional $f_b:\S=\bR^R \to \bR$, and we can regard the value $f_b(\sigma(u))$ as a stress on the brace $b$ induced by the sheer $\sigma(u)$. The raison d'etre for $SS(\B)$ is the following lemma.

\begin{lem}\label{l:raisondetre} (i) The velocity field
$u\in \M$ restricts to an infinitesimal flex of the braced 4-cycle for $b$ if and only if $f_b(\sigma(u))=0$. 

(ii) $\G(\B)$ is infinitesimally rigid if and only if 
$\rank SS(\B) = |R|-1$.
\end{lem}

\begin{proof}(i) $f_b(\sigma(u))=0$ if and only if the 2 ribbons for $b$ have the same sheer values, that is if and only if the 2 sheers of the bar 4-cycle defined by $u$ are equal. For the Euclidean norm this means that $u$ restricts to a rigid infinitesimal motion of the bar 4-cycle for $b$.

(ii) By (i) $u$ is an infinitesimal flex of $\G(\B)$ if and only if $\sigma(u) \in \ker SS(\B)$. Since $\sigma$ is onto, by Lemma \ref{l:sigmaLemma} infinitesimal rigidity holds if and only if $\rank SS(\B) = |R|-1$.
\end{proof}

\noindent {\bf Proof of the Bolker-Crapo theorem.} Let us show once again that the graph condition is sufficient, using Lemma \ref{l:sigmaLemma} in place of Lemma \ref{l:ABCDlemma}.
If $u$ is an infinitesimal flex in $\M$ with sheer field $\sigma(u)$, and if $\sigma(u)$ belongs to $\ker SS(\B)$ then the values of $\sigma(u)$ on the 2 ribbons for each brace $b$ are equal. Thus if the brace-ribbon graph $H$ is connected and spanning we deduce that if $\sigma(u) \in\ker SS(\B)$ then $\sigma(u)$ is constant. Thus $\sigma(u)= \lambda \ul{1}$ for some $\lambda\in \bR$ and Lemma \ref{l:sigmaLemma}, $\G$ is infinitesimally rigid.

Suppose that $H$ is not connected. Then 
there is a permutation of rows and columns of ribbons  resulting in a braced grid with the braces in blocked form, by which we mean the following: there exist $1\leq m_1< m, 1\leq n_1< n$ such that if the bar 4-cycle with label $(k,l)$ is braced then either  $k\leq m_1$ and $l\leq n_1$, or $k> m_1$ and $l> n_1$. Also there is at least 1 braced square in each block. Such braced grids are not infinitesimally rigid since there exists an infinitesimal flex which fixes one block and gives an infinitesimal rotation of the other block. It follows from Lemma \ref{l:raisondetre} that the original unpermuted braced grid also fails to be infinitesimally rigid.  Finally, note that if $H$ is not spanning then there is a ribbon which is free of braces and so the braced grid is not infinitesimally rigid. 
\medskip


\begin{rem}
Every real finite matrix  $A$ defines a matroid, $M(A)$ say, on a finite set $X$ with cardinality equal to the number of rows of the matrix.
The independent sets of the matroid are the subsets of $X$ corresponding to sets of rows which are linearly independent.  Bolker and Crapo use the term \emph{vector geometry} for the matroid defined by a set of row vectors. The Bolker-Crapo theorem for a  braced grid $\G$ may be expressed in an alternative form as a  matroid isomorphism: 
the \emph{stress-sheer matroid} for stress-sheer matrix for the completely braced grid is isomorphic to the graphic matroid of $K_{m,n}$, the braces graph for the completely braced $m \times n$  grid. Indeed, the $mn \times (m+n)$ stress-sheer matrix provides a linear representation of this graphic matroid.

For discussions of the $\|\cdot\|_2$-rigidity of braced grids in $\bR^3$  see Bolker \cite{bol} and Recksi \cite{rec}. In particular Recksi obtains a matroidal condition which is necessary for rigidity. As far as the author is aware there is no combinatorial characterisation of infinitesimal rigidity known for 3D braced grids. (The articles \cite{bol-cra-siam}, \cite{bol-siam} are minor variations of  \cite{bol-cra}, \cite{bol}.)

\end{rem}

\begin{rem} The infinitesimal rigidity of a general bar-joint framework, with nonzero edge lengths, is determined by its $|E| \times d|V|$ \emph{rigidity matrix} $R(G,p)$ \cite{asi-rot}. The combinatorial characterisation of Euclidean plane bar-joint frameworks $(G,p)$ whose joints are \emph{generically positioned} is due to Pollaczek-Geiringer \cite{pol} and Laman \cite{lam}. (See also Bernstein \cite{ber} for a recent matroidal proof.) The condition 
is that $G$ should contain a spanning graph $H$ which is \emph{$(2,3)$-tight} in the sense that $2|V|-|E|=3$ and for each subgraph $H'$, with at least 2 vertices, $2|V'|-|E'|\geq3.$ 
Since the joints of
a grid of squares are not generically placed (the set of $2(n+1)(m+1)$ coordinates is not an algebraically independent set) 
special position considerations are required for a proof of the Bolker-Crapo theorem. In particular if $\G(\B)$ is the 3 by 3 braced grid with the corner squares and the central square braced then this is a bar-joint framework $(G,p)$ which is infinitesimally flexible. On the other hand $G$ is $(2,3)$-tight and so a generic placement $(G,p')$ is infinitesimally rigid.
\end{rem}


\section{Non-Euclidean braced grids}\label{s:nonEthm} 

Let $\G=(G,p)$ be a bar-joint framework in $\bR^2$, with $n$ vertices, and let $\|\cdot\|$ be a general norm for $\bR^2$.
An \emph{infinitesimal flex} with respect $\|\cdot\|$ is vector (or velocity) field $u = (u_1, \ldots, u_n)$, with $u_i\in \bR^2$ for all $i$, such that
\[
\|(p_i+tu_i)-(p_j+tu_j)\|-\|p_i-p_j\|=o(t), \,\,\,\,\,\,\,\, \mbox{ as } t\to 0+,
\]
for each edge $v_iv_j$ of $G$.
Assume that the unit sphere of the normed space $(\bR^2,\|\cdot\|)$ is differentiable at the point $p_2/\|p_2\|$. In this case if
$p_1=(0,0)$  then the velocity field $(0,u_2)$ for the bar $p_1p_2$ is a $\|\cdot\|$-infinitesimal flex  if and only if $u_2$ is zero or is tangential to this unit sphere at the point $p_2$. 

An \emph{infinitesimal rigid motion} or \emph{rigid motion flex} of a bar-joint framework $\G$ in $(\bR^2, \|\cdot\|)$ may be defined as an infinitesimal flex which extends to an infinitesimal flex of any framework  $\G'$ in $(\bR^2, \|\cdot\|)$ which contains $\G$. 
As in the Euclidean setting, a bar-joint framework is \emph{infinitesimally rigid} in $(\bR^2, \|\cdot\|)$ if all its infinitesimal flexes are rigid motion flexes. 

Let us say that an $m \times n$ braced or unbraced grid has \emph{inclination} $\alpha \in [0,\pi/2)$ if it is associated with the image of the closed set $C(m,n)$ under a translation and positive rotation by $\alpha$. Such a framework is equivalent to
its uninclined variant in the normed space $(\bR^2, \|\cdot\|')$, where
the unit sphere for $\|\cdot\|'$ is the counterclockwise rotation by $\alpha$ of the  $\|\cdot\|$-sphere.

Assume for the remainder of this section that the norm $\|\cdot\|$ is {differentiable} and strictly convex, so that that the unit sphere is a curve which has well-defined tangents at all points. For a general non-Euclidean norm the space of rigid motion flexes is the 2-dimensional space of translation flexes \cite{kit-pow}. Under our assumption on the norm this follows on showing that there exists a  framework $(K_4,p)$ whose only infinitesimal flexes are infinitesimal translations.

\subsection{Tangent vectors and 4-fold symmetry.} 
 As in the Euclidean case, there is a natural basis for the space of all $\|\cdot\|$-infinitesimal flexes of an unbraced grid $\G$ in $(\bR^2, \|\cdot\|)$ with inclination $\alpha$. This is the set
$\{u_{l,\alpha}:l\in L\}$ where we label the (now possibly inclined) lines of the grid as before, where $u_{l,\alpha}$ is supported by the joints of the line. The individual velocities at these joints are equal to the unique unit vector which is tangential to a $\|\cdot\|$-sphere centred on an orthogonal line through such a joint, with this centre to the left of the line with respect to its positive direction. This positive direction corresponds to increasing coordinates when the grid is not inclined.
 With $\alpha$ fixed we denote this tangent vector as $u_x$ or $u_y$, according to whether $l$ belongs to $L_x$ or $L_y$.
The vector $u_x$ and an associated basis vector field $u_{l, \alpha}$ are illustrated in the diagrams of Figure \ref{f:inclinedbasis}. 

The first diagram of Figure \ref{f:inclinedbasis} shows tangent vectors to the $\|\cdot\|_2$-sphere and a $\|\cdot\|$-sphere for the particular radial angle, $\theta = \alpha$, measured from a downward radius. For a general radial angle $\theta$ these tangents have angles $\theta$ and $\tau(\theta)$ respectively, for some strictly increasing function
 $\tau: \theta \to \tau(\theta), 0\leq \theta <\pi$. 

\begin{center}
\begin{figure}[ht]
\centering
\includegraphics[width=3.5cm]{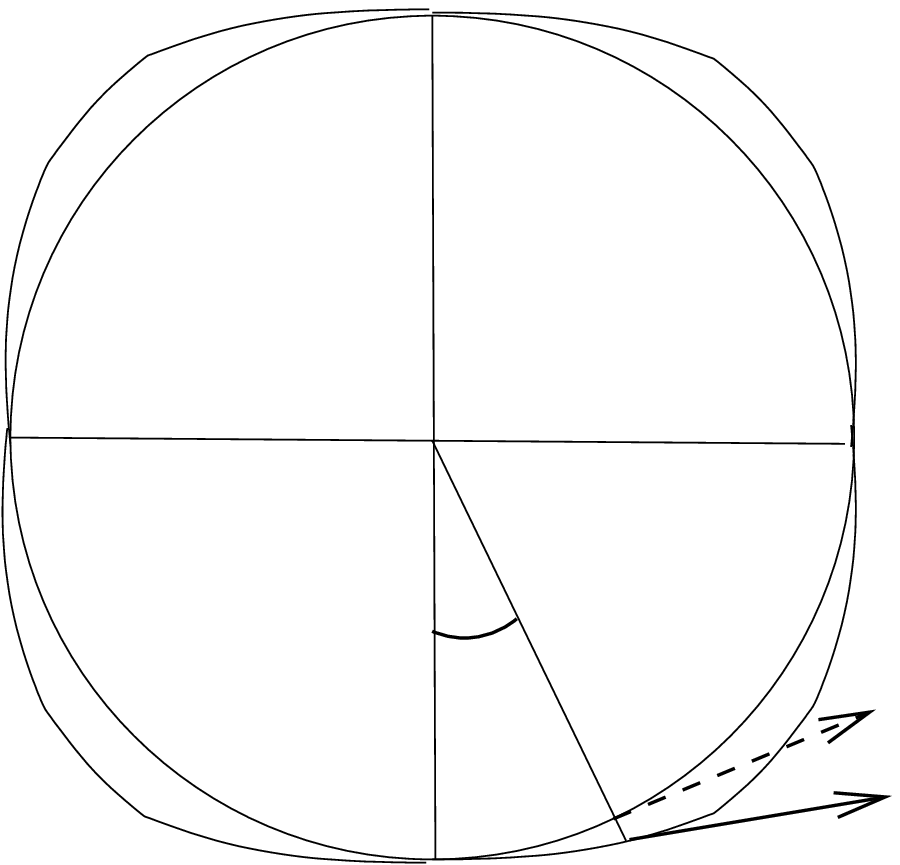}\quad \quad \includegraphics[width=4cm]{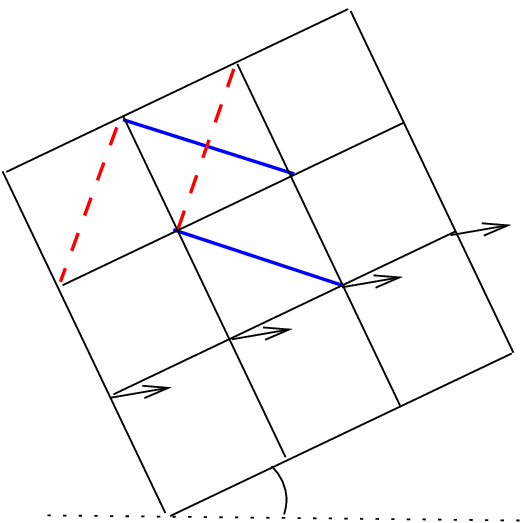}
\caption{(i) Tangent vectors 
to the unit spheres for $\|\cdot\|_2$ and $\|\cdot\|$, and (ii) a corresponding basis element $u_{l, \alpha}$ for the inclined grid framework.}
\label{f:inclinedbasis}
\end{figure}
\end{center}

The normed space $(\bR^2, \|\cdot\|)$ is said to be \emph{4-fold symmetric},
or have 4-fold symmetry, if rotation by $\pi/4$ is an isometry. This is equivalent to the periodicity condition
$\tau(\theta) = \tau(\theta+\pi/4)$, for $0\leq \theta <\pi/4$.
In particular the usual $p$-norms $\|\cdot\|_p, 1 <p<\infty,$ are differentiable strictly convex norms with $4$-fold symmetry.  


\subsection{Brace parameters and stress-sheer matrices}\label{ss:ssmatrix}
We now determine an appropriate stress-sheer matrix for a possibly inclined braced grid $\G(\B)$ in the normed space $(\bR^2, \|\cdot\|)$. This is given in terms of positive real numbers $\lambda, \lambda'$ which we refer to as the \emph{brace parameters} for $\G(\B)$. 

As in the Euclidean case write $\M=\bR^L$ for the vector space of coefficients  representing infinitesimal flexes of the unbraced grid $\G$ with respect to the line-labelled basis $\{u_{l,\alpha}:l\in L\}$.
The space $\S= \bR^R$, of ribbon sheers, and the surjection $\sigma: \M \to \S$ are defined in terms of this basis as before. Thus
\[
\sigma(u)(\rho) = 
\sgn(l_1;\rho)d_{l_1} + \sgn(l_2;\rho)d_{l_2},
\quad \mbox{ for }\rho \in R, ~~~ u = \sum_{l\in L} d_lu_{l,\alpha}.
\]
Once again $\ker \sigma$ is the 2-dimensional space of translation rigid motion flexes.
For a given ribbon $\rho$ let us write $u_\rho$ for the specific infinitesimal flex $u$, of the unbraced grid, with $\sigma(u)(\rho)=1$ and support on the boundary lines of the ribbon. This means that $d_l=\sgn(l;\rho)1/2$ for these  boundary lines and $d_l=0$ otherwise.

The braces of $\B$ are now of 2 types, indicated in Figure \ref{f:inclinedbasis}(ii) as solid (blue) and dashed (red) line segments, and we write $b, b'$ for a pair of these two types for a particular bar 4-cycle.
Consider the $x$-ribbon sheer, $\rho_x$ say, and the $y$-ribbon sheer, $\rho_y $ say,  which are the two basis elements of $\S= \bR^R$ associated with a pair of braces  $b, b'$. By the strict convexity of the norm the vectors $u_x, u_y$ are linearly independent. Also, we claim that there are unique linear combinations of the form
\begin{equation}\label{e:braceparams}
u_b= \lambda u_{x} +u_{y}, \quad \quad u_{b'}=-\lambda'u_{x} +u_{y},\quad \lambda, \lambda' >0,
\end{equation}
such that $u_b$ (resp. $u_{b'}$) is in the tangential direction for $b$
(resp $b'$). That is,  $u_b$ (resp. $u_{b'}$) is parallel to $\tau(\alpha+\pi/4)$ (resp. $\tau(\alpha+\pi/2)$). To see this claim 
it suffices to show that these tangential directions for $b, b'$ do not coincide with the direction of $u_y$ or $-u_y$. However, by the strict convexity of the norm the direction of $u_b$ (resp. $u_{b'}$) is strictly intermediate between the directions of $u_x$ and $u_y$ (resp. $u_x$ and $-u_y$) and so the claim follows. See also Figure \ref{f:braceflex}.
\begin{center}
\begin{figure}[ht]
\centering
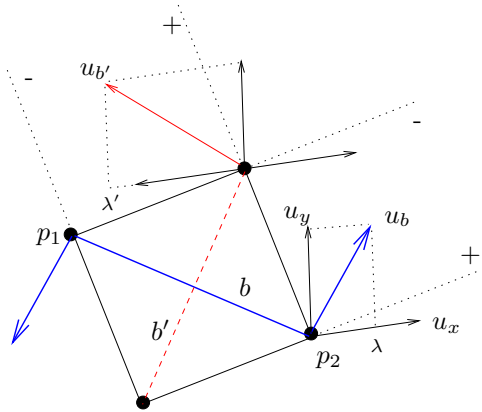
\caption{The brace parameters $\lambda, \lambda'$ are determined by the tangent vectors for $b, b'$ and the tangent vectors $u_x, u_y$.}
\label{f:braceflex}
\end{figure}
\end{center}


The \emph{non-Euclidean stress-sheer matrix} of $\G(\B)$ for the strictly convex differentiable norm $\|\cdot\|$, 
denoted $SS_{\lambda,\lambda'}(\B)$, is the $|\B|\times |R|$ matrix where the row labelled by a solid (blue) brace (resp. dashed (red) brace) has zero entries except for entries of 1 and $-\lambda$ (resp. 1 and $-\lambda'$) for the columns for the ribbons $\rho_x, \rho_y$ associated with the brace.
We view the row for each brace $b\in \B$ as determining a \emph{brace functional} $f_b:\S=\bR^R \to \bR$. The rationale for the definition of 
$SS_{\lambda,\lambda'}(\B)$ is that the velocity field $u\in \M$ restricts to an infinitesimal flex of the bar 4-cycle with added brace $b$ (resp. $b'$) if and only if $f_b(\sigma(u))=0$ (resp. $f_{b'}(\sigma(u))=0$).
To see this in the case of $f_b$ let $u$ be a velocity field with restriction velocity field $(0,u_2)$ for the bar $p_1p_2$ indicated in Figure \ref{f:braceflex}. This is a flex of the bar if and only if $u_2= cu_b = c(\lambda u_x+u_y)$ for some $c\in \bR$. On the other hand note that from the definition of $\sigma$ this is equivalent to 
\[
\sigma(u)(\rho_x) = c\lambda,\quad  \sigma(u)(\rho_y)=c,
\] 
for some $c$, which is the same as $f_b(\sigma(u))=0$. The argument for $f_{b'}$ is similar.

Let $\ul{1}_x$ (resp. $\ul{1}_y$) be the sum of the basis elements for the $x$-ribbons (resp. $y$-ribbons).
If the brace parameters agree then the ribbon sheer field $s=\lambda\ul{1}_x+\ul{1}_{y}$ lies in the nullspace of the stress-sheer matrix and so $u$ is a non rigid motion infinitesimal flex if $\sigma(u)=s$, assuming the norm is not Euclidean. It follows that the fully braced grid (of the same inclination) is not infinitesimally rigid. Thus we have obtained the next lemma.

\begin{lem}\label{l:braceparams} The following are equivalent for a non-Euclidean, differentiable, strictly convex norm $\|\cdot\|$ and an angle $\alpha \in [0,\pi/2)$.

(i) The doubly braced $1 \times 1$ square grid of inclination $\alpha$ is infinitesimally flexible in $(\bR^2, \|\cdot\|)$.

(ii) The brace parameters $\lambda, \lambda'$ for $\|\cdot\|$ and $\alpha$ are equal. 

(iii) A maximally braced $m \times n$ square grid framework of inclination $\alpha$ is infinitesimally flexible in $(\bR^2, \|\cdot\|)$.
\end{lem}

It is straightforward to see that a monochrome cycle of edges in the braces graph gives a circuit in the matroid $M(SS_{\lambda,\lambda'}(\B))$. More generally we have the following.

\begin{lem}\label{l:indptcycles}
Let $\lambda\neq \lambda'$. A cycle of edges in $K_{m,n}^2$ gives a circuit in the matroid $M(SS_{\lambda,\lambda'}(\B_{\rm max}))$ if and only if its colouring is a dependent colouring.
\end{lem}

\begin{proof}
For notational convenience consider a coloured 6-cycle $e_1, \dots ,e_6$  in $K^2_{m,n}$. The nonzero entries of the rows of $SS_{\lambda,\lambda'}(\B_{\rm max})$ determine a $6 \times 6$ matrix of the form
\[
\left[\begin{array}{cccccc}
 1 & 0&0 &\lambda_1& 0&0 \\
 1 & 0&0 &0&\lambda_2& 0 \\
 0 & 1&0 &0&\lambda_3& 0 \\
 0 & 1&0 &0&0&\lambda_4 \\
 0 & 0&1 &0&0&\lambda_5 \\
 0 & 0&1 &\lambda_6&0&0
 \end{array}\right].
 \] 
with $\lambda_i\in \{-\lambda, -\lambda'\}$, for $1\leq i \leq 6$. Under row operations this matrix is equivalent to 
\[
\left[\begin{array}{cccccc}
 1 & & &\lambda_1& & \\
  & 1& & &\lambda_3&  \\
  & & 1 &&&\lambda_5 \\
 0 & & &-\lambda_1&\lambda_2&0 \\
  & 0& &0&-\lambda_3&\lambda_4 \\
  & &0 &\lambda_6&0&-\lambda_5
 \end{array}\right]
 \] 
and this has zero determinant if and only if 
\[
\lambda_1\lambda_3\lambda_5 =\lambda_2\lambda_4\lambda_6.
\] 
Similarly, a coloured $2n$-cycle gives a dependent set in the stress-sheer matroid if and only if the associated odd and even products are equal. Since $\lambda, \lambda' >0$ this is the case if and only if
the multiplicities of $\lambda$ in the odd and even products are the same, as required.
\end{proof}

A subgraph of the bicoloured graph $K_{m,n}^2$ is an \emph{independent  cycle-rooted tree} if it is a cycle-rooted tree whose unique cycle is \emph{independent}. A  \emph{independent cycle-rooted forest} is a subgraph whose components are independent cycle-rooted trees. In an extreme case each component could be a bicoloured cycle with 2 vertices.

\begin{lem}\label{l:circuitsETC} 
Let $\lambda\neq \lambda'$.
Then the maximal independent sets of the matroid $M(SS_{\lambda,\lambda'}(\B_{\rm max}))$ are the independent cycle-rooted forests which are spanning subgraphs of $K_{m,n}^2$.
\end{lem}

\begin{proof} An independent cycle-rooted forest which is spanning has $m+n$ edges and so, by the previous lemma and the fact that 
$SS_{\lambda,\lambda'}(\B_{\rm max})$ has $m+n$ columns, it follows that
it is a maximal independent set in the matroid. 
On the other hand if $F$ is a maximal independent set of edges in the matroid then it has $m+n$ edges. Each component cannot be a tree for otherwise an appropriately coloured edge could be added to create an independent cycle-rooted tree. Thus each component contains an independent cycle rooted-tree and by the previous lemma must be equal to it.
\end{proof}

\begin{lem}\label{l:raisondetreNONEUCL}
Let $\G(\B)$ be a braced grid with brace parameters $\lambda \neq  \lambda'$. Then
$\G(\B)$ is infinitesimally rigid
if and only if $SS_{\lambda,\lambda'}(\B)$ has rank $m+n$.
\end{lem}
\begin{proof}
It follows from the definition of the brace functionals that a velocity field $u\in \M$ restricts to an infinitesimal flex of the bar 4-cycle with added braces $b, b'$ if and only if $f_b(\sigma(u))=0$ and $f_{b'}(\sigma(u))=0$. Also $\ker \sigma$ is the space of infinitesimal translations.
Thus $\G(\B)$ is infinitesimally rigid if and only if 
$\ker SS_{\lambda,\lambda'}(\B) =\{0\}$.
\end{proof}

\noindent{\bf The proof of Theorem \ref{t:nonEthm}.}
Suppose that $\G(\B)$ is infinitesimally rigid.
Then the brace parameters must be distinct, by Lemma \ref{l:braceparams}, and so
$\rank SS_{\lambda,\lambda'}(\B) =m+n$ by Lemma \ref{l:raisondetreNONEUCL}. Thus there exists an independent set of $m+n$ rows.
By Lemma \ref{l:circuitsETC}  these rows correspond to the edges of a independent cycle-rooted forest, and so the bicoloured graph condition in (ii) follows.

On the other hand if $H\subset K_{m_1, n_2}^2$
is a spanning independent cycle-rooted tree then by Lemma \ref{l:circuitsETC} the rows of $SS_{\lambda,\lambda'}(\B)$
are a maximal linearly independent subset and $\rank SS_{\lambda,\lambda'}(\B)= m+n$, completing the proof.
\medskip


\subsection{Irregularly spaced grids}
Consider the bar-joint frameworks $\G(\B)$ arising from an irregularly spaced $m \times n$ grid framework $\G$, aligned with the coordinate axes, and a set $\B$ of diagonal braces. Assuming that the underlying norm is differentiable and strictly convex there is a set of positive brace parameters, $\Lambda$ say, together with a 
stress-sheer matrix $SS_\Lambda(\B)$ defined as before. Lemma \ref{l:indptcycles} suggests the following gain-graph formalism, with edge-labelling by elements of the abelian group $\bR_+$.

Denote the braces and their parameters as $b_e$ and $\lambda_e$, where $e$ is a coloured edge of the braces graph. This graph is bipartite with vertices $v_1, \dots , v_m, w_1,\dots , w_n$. Choose an associated edge directedness, so that an edge is positively directed from its $v$-labelled vertex to its $w$-labelled vertex. Also, define the directed edge gain map $\gamma: E(K_{m,n}^2) \to\bR_+$ where $\gamma(e)=\lambda_e$ when $e$ is positively directed.
Following Lemma \ref{l:indptcycles} we see that a directed cycle $c$ in the braces graph corresponds to a dependent set in $M(SS_\Lambda(\B))$ if and only if $\gamma(c)=1$, that is, if and only if the cycle has no gain. As before we say that the cycle is \emph{dependent} in this case and \emph{independent} otherwise. The statement and proof of Theorem \ref{t:nonEthm} generalise, with no essential changes, to irregularly spaced grids.

\begin{rem} The matroid theory of signed graphs and gain graphs has been developed extensively by Zaslavsky. See for example \cite{zas-1}, \cite{zas-2}. In fact Bolker \cite{bol} made use of signed graphs to determine the circuits for the 3-dimensional braced cube grid rigidity matroid in graphic terms.

 For a (monochrome) simple graph the matroid whose independent sets are the cycle-rooted forests is known as the bicycle matroid \cite{whi-union}. Also cycle-rooted spanning forests appear in  expansion formulae for the determinant of the combinatorial Laplacian on a graph \cite{ken}. 
\end{rem}

We now consider the special case of the classical $p$-norms, $\|\cdot\|_p$ with $\|(x,y)\|_p^p = |x|^p+|y|^p$.

\begin{thm}\label{t:lpgrids}
Let $1< p < \infty$ with $p\neq 2$ and let $\G(\B)$ be a diagonally braced $m\times n$ grid of congruent orthogonal rectangles with inclination $\alpha \in [0,\pi/2)$.

(i) If the rectangles are squares then $\G(\B)$ is infinitesimally flexible in $(\bR^2, \|\cdot\|_p)$.

(ii) If the rectangles are not squares then $\G(\B)$ is infinitesimally rigid in $(\bR^2, \|\cdot\|_p)$ if and only if the braces graph contains an independent cycle-rooted forest which is spanning. 
\end{thm}

\begin{proof}
(i) The classical $p$-norms for $1< p < \infty$ are differentiable, strictly convex and 4-fold symmetric. By $4$-fold symmetry the vectors $u_x$ and $u_y$ are orthogonal for any inclination value $\alpha$. Also by $4$-fold symmetry, the braces tangents $u_b$  and $u_{b'}$ for a grid of squares are orthogonal for any value of $\alpha$. In view of this double orthogonality, Equation \ref{e:braceparams} has solutions $\lambda = \lambda'$, and (i) follows.

(ii) When the rectangles are not squares then their diagonals subtend an angle $0<\beta<\pi$ with $\beta \neq\pi/2$.  Also, the tangent function $\tau(\theta)$ for $\|\cdot\|_p$ has the property that 
$\tau(\theta+\beta)\neq \tau(\theta)+\pi/2$, for
$ \beta\neq \pi/2$, and so the braces tangents are not orthogonal  for any value of $\alpha$. By the brace parameter equations (\ref{e:braceparams}), $\lambda\neq \lambda'$ (since $u_x, u_y$ are orthogonal) and so Lemma \ref{l:braceparams} applies.
\end{proof}

\begin{rem}\label{r:22tight}
A graph $G$ is $(2, 2)$-tight if  $2|V|-|E|=2$ and for each subgraph $G'$ we have $2|V'|-|E'|\geq 2.$ 
The characterisations in Dewar \cite{dew}, and Kitson and Power \cite{kit-pow}, show that the existence of a $(2,2)$-tight spanning subgraph of $G$ is necessary and sufficient for
the infinitesimal rigidity of a bar-joint framework $(G,p)$ which is ``sufficiently generic". Under our assumptions for the underlying norm ``sufficiently generic"  corresponds to the non-Euclidean variant of the Euclidean rigidity matrix $R(G,p)$ having maximum rank, in which case the framework $(G,p)$ is said to be \emph{regular}. 
It follows from our analysis that the braced square grid $\G(\B)$ is regular if and only if the brace parameters are distinct and the cycles of the cycle-rooted forest are independently coloured.

In the case of the nondifferentiable norm  $\|(x,y)\|_\infty=$max$\{|x|,|y|\}$ the tangent function $\tau(\theta)$ is not defined for $\theta=\pi/4, 3\pi/4$. Let us say that a braced rectangle grid $\G(\B)$ is \emph{well-positioned for $\|\cdot\|_\infty$}
if $\alpha\neq  \pi/4, 3\pi/4$ and the two brace directions have well-defined tangent directions which are orthogonal. Then one can show, as before, that if $\G(\B)$ is well-positioned then it is infinitesimally rigid for $\|\cdot\|_\infty$ if and only if the braces graph is spanning and each component has an independent cycle. See also
the general characterisation of infinitesimal $\|\cdot\|_\infty$-rigidity, for well-positioned, regular frameworks, given in Kitson and Power \cite{kit-pow}. 
\end{rem}

Let us also note the curiosity of the special rigidity requirements for a braced square grid $\G(\B)$, with zero inclination, with respect to the non-differentiable norm $\|\cdot\|_\infty$. By the square geometry of the unit sphere, a diagonal brace, say $p_1p_2$ with $p_1=(0,0), p_2=(1,1)$, has the one-sided infinitesimal flexes $u=((0,0),(-1,0))$ and $u=((0,0),(0,-1))$. Also, the infinitesimal flexes no longer form a vector space.

\begin{prop}
An axis-aligned braced grid of squares is infinitesimally rigid with respect to $\|\cdot\|_\infty$ if and only if each vertex of the braces graph is incident to blue and red edges.
\end{prop}

The graph condition is equivalent to requiring that each ribbon contains at least one brace of each type. This is necessary and sufficient to rule out a sheering flex of $\G(\B)$ associated with the ribbon. Note, in particular that a ribbon bar-joint framework containing braces of only one type (colour) has a one-sided sheering flex. The proposition follows readily from this.

For some further discussions of non-Euclidean frameworks see also 
Dewar \cite{dew}, Kitson \cite{kit}, Kitson, Nixon and Schulze \cite{kit-nix-sch}, and Nixon and Power \cite{nix-pow}.

\section{Infinite braced grids}\label{s:infinite}

The definitions  of infinitesimal flex and infinitesimal rigidity for a countably infinite bar-joint framework are the same as those for a finite bar-joint framework \cite{kit-pow}. 
Let $\G_\infty$ be the infinite unbraced grid bar-joint framework in the usual Euclidean space $\bR^2$ with joints located at points with integer coordinates. Once again there is a distinct set of infinitesimal flexes, $\{u_l: l\in L\}$, which is indexed by the lines of $\G_\infty$. Also they form a generalised basis in the following sense.

\begin{lem}\label{l:flexspaceofGrid} Every infinitesimal flex $u:\bZ^2 \to \bR^2$ of the grid framework $\G_\infty$ has a unique representation
\[
u =\sum_{l\in L}   d_lu_l = \sum_{l\in L_x} a_lu_l +\sum_{l\in L_y} b_lu_l.
\]
\end{lem}


The proof of the Bolker-Crapo Theorem carries over to give the following.

\begin{thm}\label{t:infiniteNonE} An infinite braced grid $\G_\infty(\B)$ in $(\bR^2, \|\cdot\|_2)$ is infinitesimally rigid if and only if the braces graph is a connected spanning subgraph of $K_{\infty,\infty}$.
\end{thm}

\begin{proof}[Proof sketch]
Once again there are brace functionals $f_b: R\to \bR$ and $\G_\infty(\B)$ is infinitesimally rigid if and only if the intersection of the nullspaces $\ker f_b$, for $ b \in \B$, is equal to $\bR\ul{1}$. Here $\ul{1}$ is the sheer field function, on the infinite set of ribbons, $R$, which is identically equal to 1. It follows that infinitesimal rigidity is preserved on permuting the (infinite) rows or columns of the braced grid. If the braces graph is not connected, with at least 2 infinite components, then, by permuting, we may assume that the braces are on bar 4-cycles in either the first or third quadrant of $\bZ^2$. Thus  $\G_\infty(\B)$ fails to be infinitesimally rigid since there is a nonzero flex fixing the bar 4-cycles in the first quadrant. A similar arguments applies whatever the cardinality of the components. The braces graph must be spanning, or else there is a ribbon that is free of braces and so there is a sheering infinitesimal flex. Thus the graph condition is necessary.

By Lemma \ref{l:ABCDlemma} the graph condition implies that
the infinitesimal flex space of $\G_\infty(\B)$ is equal to that of
$\G_\infty(\B_{\rm max})$ and so the sufficiency direction follows.
\end{proof}



In a similar way the equivalence in Theorem \ref{t:nonEthm} extends to  infinite braced grids $\G_\infty(\B)$.

\begin{eg}
The subgraph $H$ of $K_{\infty, \infty}$ indicated in Figure \ref{f:nonSeqRig}
is a braces graph  for a Euclidean infinite braced grid $\G_\infty(\B)$ where the first (resp. second) row of vertices corresponds to a consecutive ordering of the horizontal (resp. vertical) ribbons. By Theorem \ref{t:infiniteNonE} 
this braced grid is infinitesimally rigid. On the other hand it is straightforward to show that it contains no finite subframework, with more than 2 brace bars, which is infinitesimally rigid.
\begin{center}
\begin{figure}[ht]
\centering
\includegraphics[width=8cm]{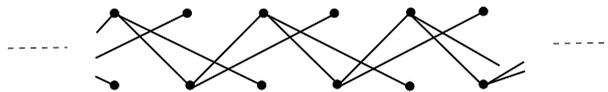}
\caption{A braces graph subgraph of $K_{\infty, \infty}$.}
\label{f:nonSeqRig}
\end{figure}
\end{center}
We note that this phenomenon is not possible for an infinite \emph{generic} bar-joint framework $(G,p)$ in $\bR^2$. Such a framework is infinitesimally rigid if and only if $(G,p)$ is \emph{sequentially infinitesimally rigid} (Kitson and Power \cite{kit-pow-infinite}). This means that there exists an increasing chain of infinitesimally rigid finite subframeworks $(G_n,p)$ with $G$ equal to the union of the $G_n$.
\end{eg}

\end{document}